\documentclass[10pt,reqno]{amsart}
\usepackage{amssymb,amsmath,amsthm,enumerate}

\newcommand{\Z}{{\mathbb Z}}
\newcommand{\R}{{\mathbb R}}

\newcommand{\LR}{{L^2({\mathbb R})}}
\newcommand{\Ltwo}{L^2([0,1])}

\newcommand{\cl}{{\rm cl}_{L^2}\,}
\newcommand{\nep}{n_{\epsilon}}

\theoremstyle{plain}
\newtheorem{theorem}{Theorem}
\newtheorem{lemma}[theorem]{Lemma}

\theoremstyle{definition}

\theoremstyle{remark}
\newtheorem{rem}[theorem]{Remark}
\def\hypergeom#1#2#3#4#5{{}_#1 F_{#2}\left({#3\atop#4}; \ #5\right)}

\title{A hypergeometric basis for the Alpert multiresolution analysis}
\subjclass[2010]{42C40, 41A15, 33C20}
\keywords{Multiwavelets, Hypergeometric functions, Orthogonal polynomials}

\author[J.~Geronimo]{Jeffrey~S.~Geronimo}
\address{JSG, School of Mathematics, Georgia Institute of Technology,
Atlanta, GA 30332--0160, USA}
\email{geronimo@math.gatech.edu}
\thanks{JSG is partially supported by Simons Foundation Grant \#210169.}

\author[P.~Iliev]{Plamen~Iliev}
\address{PI, School of Mathematics, Georgia Institute of Technology,
Atlanta, GA 30332--0160, USA}
\email{iliev@math.gatech.edu}
\thanks{PI is partially supported by Simons Foundation Grant \#280940.}
\begin{document} 
\date{November 12, 2014}

\begin{abstract}
We construct an explicit orthonormal basis of piecewise ${}_{i+1}F_{i}$ hypergeometric polynomials for the Alpert multiresolution analysis. The Fourier transform of each basis function is written in terms of ${}_2F_3$ hypergeometric functions.  
Moreover, the entries in the matrix equation connecting the wavelets with the scaling functions are shown to be balanced ${}_4 F_3$ hypergeometric functions evaluated at $1$, which allows to compute them recursively via three-term recurrence relations. 

The above results lead to a variety of new interesting identities and orthogonality relations reminiscent to classical identities of higher-order hypergeometric functions and orthogonality relations of Wigner $6j$-symbols.
\end{abstract}

\maketitle

\section{Introduction}\label{se1}

Wavelet theory has had proved useful in many areas of mathematics and 
engineering such as functional analysis, Fourier analysis, and signal 
processing. The Alpert multiresolution analysis is associated with the 
spline spaces of piecewise polynomials of degree at most $n$, discontinuous at 
the integers. The generators for this multiresolution analysis are the Legendre polynomials, restricted and scaled to $[0,1)$ and set to zero otherwise. 
Using the symmetry inherent in this system, Alpert \cite{ala,alb} constructed a set of wavelet functions and then 
used them to analyze various integral operators (see also \cite{abcr}). This 
basis has been used in \cite{dd} with the moment interpolating technique to 
construct smooth multiwavelets. An interesting problem from both computational and theoretical point of view is to find a 
wavelet basis  which can be  written in terms of explicit formulas.  

In \cite{gm} an analysis was performed on the coefficients in the refinement equation satisfied by the modified Legendre 
polynomials and it was shown that the entries in these matrices could be written as multiples of certain generalized Jacobi polynomials evaluated 
at $1/2$ and that these entries satisfy generalized eigenvalue equations. Moreover a new basis of wavelets was implicitly introduced through the matrix equation connecting the wavelets to the scaling functions by considering  upper triangular matrices with positive diagonal entries. 

In the present paper we provide explicit formulas for these wavelets, their Fourier transforms, and related matrix coefficients in terms of hypergeometric functions. We give  a different construction and direct proofs of the characteristic properties of this new basis thus making the paper self-contained. Our results imply new identities between higher-order hypergeometric functions and suggest interesting connections to representation theory.

The paper is organized as follows. In Section~\ref{se2} we set the notation  and review the elements of multiresolution analysis needed for the sequel. In Section~\ref{se3} we postulate the orthogonality and symmetry conditions which characterize the wavelet functions. We show that these properties are satisfied by a sequence of piecewise polynomials supported on $(-1,1)$, which on $(-1,0)$ and $(0,1)$ can be written as ${}_{i+1}F_{i}$ hypergeometric functions. We also exhibit families of differential equations satisfied by these functions. In Section~\ref{se4} we explain the differences between the wavelets constructed here and the ones introduced by Alpert. We also prove that the entries in the matrices relating our wavelets to the scaling functions can be written as balanced ${}_4 F_3$ hypergeometric functions evaluated at $1$. In particular, these formulas imply that the matrices are upper triangular with positive diagonal entries, thus relating the formulas here to the implicit construction in \cite{gm}. We also indicate how these results are reminiscent to classical identities of higher-order hypergeometric functions and orthogonality relations of Wigner $6j$-symbols. In Section~\ref{se5} we derive a simple closed formula for the Fourier transform of these wavelets in terms of ${}_2F_3$ hypergeometric functions and we write  associated differential equations for them. Finally, in Section~\ref{se6}, we give recurrence formulas for the entries in the matrices of the wavelet equation.

\section{Preliminaries}\label{se2}
Let
$\phi_0,\dots,\phi_r$ be compactly supported $L^2$-functions, and suppose that
$V_0 = \cl{\rm span}\{\phi_i({\cdot}-j): i = 0,1,\dots,r,\ j\in\Z\}$.  Then
$V_0$ is called a {\it finitely generated shift invariant\/ {\rm (FSI)}
space}.  Let $(V_p)_{p\in\Z}$ be given by
$V_p = \{\phi(2^p{\cdot}): \phi\in V_0\}$.  Each space $V_p$ may be
thought of as approximating $L^2$ at a different resolution depending
on the value of $p$.  The sequence $(V_p)$ is called a
{\it multiresolution analysis} (MRA) \cite{da,ghm,gl} generated by {$\phi_0,\dots,\phi_r$} if
(a) the spaces are nested, $\cdots\subset V_{-1}\subset V_0\subset V_1
\subset\cdots$, and (b) the generators $\phi_0,\dots,\phi_r$ and 
their integer translates form a Riesz basis for $V_0$.  Because of (a)
and (b) above, we can write
\begin{equation}\label{nest}
V_{j+1} = V_j \oplus W_j, \qquad\forall j\in\Z. 
\end{equation}
The space $W_0$ is called the {\it wavelet space}, and if
$\psi_0,\dots,\psi_r$ generate a shift-invariant basis for $W_0$, then
these functions are called {\it wavelet functions}.
If, in addition, $\phi_0,\dots,\phi_r$ 
and their integer translates form an orthogonal basis for $V_0$,
then $(V_p)$ is called an {\it orthogonal MRA}.  Let $S_{-1}^n$ be the space of
polynomial splines of degree $n$ continuous except perhaps at the integers,
and set $V_0^n = S_{-1}^n \cap L^2(\R)$.  With $V_p^n$ as above these spaces form 
a multiresolution analysis. If $n=0$ the  multiresolution analysis obtained is 
associated with the Haar wavelet while for $n>0$ they were  introduced by 
Alpert \cite{ala, alb}. 
Let 
$$
\phi_j(t)= \hat p_j(2t-1)\chi_{[0,1)}(t)
$$
where $\hat p_j(t)$ is the  Legendre polynomial \cite{sz} of degree $j$ orthonormal on $[-1,1]$  with 
positive leading coefficient i.e. $\hat p_j(t)=k_j t^j +\text{lower degree terms}$, with $k_j>0$ and
$$
\int_{-1}^1 \hat p_j(t)\hat p_k(t)dt=\delta_{k,j}.
$$ 
These polynomials have the following representation in terms of
a ${}_2 F_1$ hypergeometric function \cite[p. 80]{sz},
\begin{equation}\label{hyperone}
\hat p_n(t) = \frac{\sqrt{2n+1}}{\sqrt{2}}
  \,\hypergeom21{-n,\ n+1}{1}{\frac{1-t}{2}},
\end{equation}
where formally,
$$
\hypergeom{p}{q}{a_1,\ \dots\ a_p}{b_1,\ \dots\ b_q}{t}
  = \sum_{i=0}^{\infty}\frac{(a_1)_i\dots(a_p)_i}{(b_1)_i\dots(b_q)_i(1)_i}t^i
$$ 
with $(a)_0=1$ and $(a)_i = a(a+1)\ldots(a+i-1)$ for $i>0$.  Since one 
of the numerator parameters in the definition of $\hat p_n$ is a negative 
integer, the series in  equation~\eqref{hyperone} has only finitely many 
terms. With the normalization taken above  $||\phi_j||^2_{\LR}=\frac{1}{2}$, 
and
\begin{equation}\label{Phi}
\Phi_n=\left[\begin{matrix}\phi_0&\cdots&\phi_n\end{matrix}\right]^T
\end{equation}
and its integer translates form an orthogonal basis for $V_0$. For the convenience in later computations we set 
\begin{equation}\label{PN}
P_n(t)=\left[\begin{matrix}\hat p_0(t)\\\vdots\\\hat p_n(t)\end{matrix}\right].
\end{equation}
Equation~\eqref{nest} implies the existence of the {\it refinement} equation,  
\begin{equation}\label{refin}
\Phi_n\left(\frac{t}{2}\right)=C^n_{-1}\Phi_n(t)+C^n_1\Phi_n(t-1),
\end{equation}
see \cite{da,gm}. In order to exploit the symmetry of the Legendre polynomials we shift $t\to t+1$
so that 
\begin{align}\label{refsym}
\Phi_n\left(\frac{t+1}{2}\right)&=P_n(t)\chi_{[-1,1)}(t)=C^n_{-1}\Phi_n(t+1)+C^n_1\Phi_n(t)\nonumber\\&=C^n_{-1}P_n(2t+1)\chi_{[-1,0)}(t)+C^n_1P_n(2t-1)\chi_{[0,1)}(t).
\end{align}

\section{Construction of the functions}\label{se3}
We want to construct a set $\{h^n_0,\ldots,h^n_n\}$ of functions such that,
\begin{enumerate}[i.]
\item On $[-1,0)$ and  $[0,1)$, $h^n_i$ is a polynomial of degree at most  $n$,
\item $\int_{-1}^1 t^s h^n_i(t)dt=0$, $0\le s\le n,\ 0\le i\le n$,
\item $\int_{-1}^1 h^n_i(t) h^n_j(t)dt=2\delta_{i,j}, \quad 0\le i, j\le n$,
\item $h^n_i(-t)=(-1)^{i+1} h^n_i(t)$ for $t\in(0,1)$, and
\item $\int_0^1 t^s h^n_i(t) dt =0,\quad s< i$.
\end{enumerate}

We will show below that these functions have a hypergeometric
representation and use them in Theorem \ref{waveident} to construct
a wavelet basis for $\LR$.

\begin{theorem}\label{eqwave} 
For $0\le t<1$, we have
\begin{align}
&h_{n-i}^n(t)=\sum_{k=0}^n d^i_{n,k} t^k, \qquad \text{if $i=0$ or $i$ odd} \label{waveodd} \\
&h_{n-i}^n(t)=\sum_{k=0}^{n-1} d^{i-1}_{n-1,k} t^k,  \qquad \text{for $i>0$ even,}\label{waveeven}
\end{align}
where
\begin{align}\label{wavecoeff}
 d^i_{n,k}&=(-1)^n\frac{\sqrt{2n-2i+1}}{(-n)_i}\frac{(-n)_k(n-i+1)_k}{(1)_k(1)_k}\nonumber\\
&\times\prod_{m=0}^{\frac{i-1}{2}}(n+k+1-2m)\prod_{m=0}^{\frac{i-3}{2}}(n-k-1-2m),\
 0\le k\le n,\ i=0,\ \text{or}\ i\ \text{odd.}
\end{align}
We extend $h_{n-i}^n(t)$ on $[-1,0)$ using  {\rm (iv)}.
\end{theorem}
In formula \eqref{wavecoeff} and throughout the paper we use the convention that an empty product is equal to $1$.

\begin{rem}
We note that conditions (i) through (iv) are the same as those imposed by Alpert in his wavelet construction. To fix a basis he imposes more vanishing moments on $[-1,1)$ (see \cite[p.~248 condition 4]{alb}). We were not able to find a hypergeometric representation for his functions.
\end{rem}

\begin{rem}
From the above formula it follows that the polynomials $\frac{(-n)_i h^n_{n-i}(t)}{\sqrt{2n-2i+1}}$ for $i=0$ or $i$ odd  and  $\frac{(-n+1)_{i-1} h^n_{n-i}(t)}{\sqrt{2n-2i+1}}$ for $i>0$ even have integer coefficients.
\end{rem}

We prove this theorem after developing a few lemmas. We begin with,
\begin{lemma}\label{lang} For $i\le n$ 
\begin{equation}\label{lg1}
\sum_{k=0}^{n}\frac{(-n)_k}{k!}\binom{n-i+k}{n-i}\frac{1}{x+k}=\frac{n! (-x+1)_{n-i}}{(n-i)!(x)_{n+1}}.
\end{equation}
In particular, when $x=n-i+1$ we have
\begin{equation}\label{ee}
\sum_{k=0}^{n}\frac{(-n)_k}{k!}\binom{n-i+k}{k}\frac{1}{n-i+k+1}=(-1)^{n+i}\frac{n! (n-i)!}{(2n-i+1)!}.
\end{equation}
\end{lemma}
\begin{proof}
Consider a polynomial $f(x)$ of degree at most $n$. The Lagrange interpolation formula at the points $x_k=-k$, $k=0,1,\dots, n$ gives
\begin{equation*}
(x)_{n+1}\sum_{k=0}^{n}\frac{(-1)^k}{k!(n-k)!}\frac{f(-k)}{x+k}=f(x),
\end{equation*}
or equivalently
\begin{equation*}
\sum_{k=0}^{n}\frac{(-n)_k}{k!}\frac{f(-k)}{x+k}=\frac{n!f(x)}{(x)_{n+1}}.
\end{equation*}
Applying the above formula with $f(x)=\binom{n-i-x}{n-i}$ we find
\begin{equation*}
\sum_{k=0}^{n}\frac{(-n)_k}{k!}\binom{n-i+k}{n-i}\frac{1}{x+k}=\frac{n! (-x+1)_{n-i}}{(n-i)!(x)_{n+1}},
\end{equation*}
which gives the result.
\end{proof}
It is easy to see,
\begin{lemma}\label{fothlemma}
For $k\ge n-i+2$ and $n-k$ odd $ d^i_{n,k}=0$.
\end{lemma}
Also
\begin{lemma}\label{sothlemma} If $i+l<n$ and $l\ge -1$ then for any  polynomial  $p(k)$ with deg $p\le i+l$,
\begin{equation}\label{fizero}
\sum_{k=0}^n\frac{(-n)_k(n-i+1)_k}{(1)_k(l+2)_k}p(k)=0.
\end{equation}
\end{lemma}
\begin{proof} By linearity, it is enough to show \eqref{fizero} for the polynomials  $p(k,r)=k(k-1)\cdots(k-r+1)$ with  $0\le r\le i+l$. If the common factor in the numerator and denominator are canceled and the change of variable $k=r+m$ made, the above sum is equal to
\begin{align*}
\sum_{k=0}^n\frac{(-n)_k(n-i+1)_k}{(1)_k(l+2)_k}p(k,r)&=c_{n,r,i,l}\sum_{m=0}^{n-r}\frac{(-n+r)_m(n-i+r+1)_m}{(1)_m(r+l+2)_m}\\&=c_{n,r,i,l}\,\hypergeom21{-n+r,\ n-i+r+1}{r+l+2}{1}\\&=c_{n,r,i,l}\frac{(-n+i+l+1)_{n-r}}{(r+l+2)_{n-r}}=0,
\end{align*}
where $c_{n,r,i,l}=\frac{(-n)_{r}(n-i+1)_{r}}{(l+2)_{r}}$. The Chu-Vandermonde identity was used to obtain the last line in the above equation.
\end{proof}

\begin{proof}[Proof of Theorem~\ref{eqwave}]
Note first that (iv) implies that (iii) will automatically hold if $i$ and $j$ have opposite parity (i.e. one of them is odd and the other one is even). This shows that the constructions of the two sets of functions $\{h^{n}_{n-i}\}_{i\text{-odd}}$ and  $\{h^{n}_{n-i}\}_{i\text{-even}}$ are independent of each other. Moreover, we have 
\begin{equation}\label{redto01}
\int_{-1}^1 t^s h^n_{n-i}(t) dt=((-1)^{n-i+s+1}+1)\int_{0}^1 t^s h^n_{n-i}(t) dt.
\end{equation}
It is easy to see now that if the functions $\{h^{n}_{n-i}\}_{i\text{-odd}}$ defined by \eqref{waveodd} and (iv) satisfy (i), (ii), (iii) and (v), then the functions $\{h^{n}_{n-i}\}_{i>0\text{ even}}$ defined by \eqref{waveeven} and (iv) also satisfy (i), (ii), (iii) and (v) (because $h^{n}_{n-i}(t)=h^{n-1}_{n-1-(i-1)}(t)$ for $i>0$ even and $t\in(0,1)$).

If $i=0$, then \eqref{waveodd} and \eqref{wavecoeff} show that $h^n_n(t)=(-1)^{n}\sqrt{2}\hat{p}_n(1-2t)$ for $t\in[0,1)$, where $\hat{p}_n(t)$ is Legendre polynomial \eqref{hyperone} and the required properties of $h^n_n(t)$ follow easily from the properties of $\hat{p}_n(t)$.

It remains to show that for $i$ odd the functions defined by \eqref{waveodd} satisfy the required conditions, so we fix $i$ odd below. 

Formula \eqref{waveodd} yields
\begin{equation}\label{nn1}
\int_0^1 t^s h^n_{n-i}(t)dt =\sum_{k=0}^n d_{n,k}^i \frac{1}{k+s+1}=c_{n,i}\sum_{k=0}^n\frac{(-n)_k(n-i+1)_k}{(1)_k(s+2)_k}q(k), 
\end{equation}
where $c_{n,i}=(-1)^n\frac{\sqrt{2n-2i+1}}{(-n)_i(s+1)!}$ and 
$$
q(k)=\prod_{j=1}^s(k+j)\prod_{m=0}^{\frac{i-1}{2}}(n+k+1-2m)\prod_{m=0}^{\frac{i-3}{2}}(n-k-1-2m).
$$
For $s<n-i$ the degree of $q(k)$ is $s+i< n$ and we deduce from Lemma~\ref{sothlemma} that the right-hand side of \eqref{nn1} is equal to zero which gives (v).

For (ii), we need to show only that $\int_{0}^1 t^s h^n_{n-i}(t) dt=0$ when $s>n-i$ and $s$ and $n-i$ have opposite parity (the rest follows from \eqref{redto01} and (v)).  For such $s$ we find
\begin{equation}\label{intort}
\int_{0}^1 t^s h^n_{n-i}(t) dt=\sum_{k=0}^n d^i_{n,k} \frac{1}{k+s+1}.
\end{equation}
Note that $k+s+1$ cancels one of the factors in
$\prod_{m=0}^{\frac{i-1}{2}}(n+k+1-2m)$ so the right-hand side of \eqref{intort} is
equal to zero from Lemma~\ref{sothlemma} with $l=-1$. This completes the proof of (ii). 

Finally, we want to show that (iii) holds when $i$ and $j$ are odd. 
Suppose first that $i$ and $j$ are odd and $i< j$. Using (v) we see that
\begin{align*}
\int_0^1h^n_{n-j}(t)h^n_{n-i}(t)dt&=\sum_{k=n-i}^n
d_{n,k}^j\int_0^1 t^k h^n_{n-i}(t)dt\\&=\sum_{n-k\ \text{even}}
d_{n,k}^j\int_0^1 t^k h^n_{n-i}(t)dt+\sum_{n-k\ \text{odd}}
d_{n,k}^j\int_0^1 t^k h^n_{n-i}(t)dt.
\end{align*}
In the sums on the last line of the above equation $n-i\le k \le n$. Since in the
second sum $n-k$ is odd, Lemma~\ref{fothlemma} shows that $d_{n,k}^j=0$
while in the first sum $k>n-i$ and $k$ and $n-i$ are of opposite parity so
the   discussion after \eqref{intort} shows that the integral is
equal to zero. We now show that
$$
\int_0^1h^n_{n-i}(t)^2 dt=1.
$$
With the substitution of \eqref{wavecoeff} we find from above that
\begin{align*}
\int_0^1h^n_{n-i}(t)^2 dt&=d^i_{n,n-i}\int_0^1
t^{n-i}h^n_{n-i}(t)dt\\&=
\frac{(-1)^n\sqrt{2n-2i+1}d^i_{n,n-i}}{(-n)_i}\sum_{k=0}^n\frac{(-n)_k(n-i+1)_k}{(1)_k(1)_k}\frac{\pi_i(k)}{n-i+k+1},
\end{align*}
where $\pi(k)$ is the polynomial in $k$ of degree $i$ given by
$$
\pi_i(k)=\prod_{m=0}^{\frac{i-1}{2}}(n+k+1-2m)\prod_{m=0}^{\frac{i-3}{2}}(n-k-1-2m).
$$
If we add and subtract $\pi_i(-n+i-1)/(n-i+k+1))$ to the summand  then using
Lemma~\ref{sothlemma} with $s=-1$ and the fact that
$(\pi_i(k)-\pi_i(-n+i-1))/(n-i+k+1))$ is a polynomial of degree $i-1$ in
$k$ we find
$$
\int_0^1h^n_{n-i}(t)^2
dt=(-1)^{i+n}\frac{(2n-i+1)!}{n!(n-i)!}\sum_{k=0}^n\frac{(-n)_k(n-i+1)_k}{(1)_k(1)_k(n-i+k+1)}.
$$
The proof now follows from equation \eqref{ee} in Lemma~\ref{lang}.
\end{proof}

\begin{rem}\label{hypergeomrep}
From the explicit formulas given in Theorem~\ref{eqwave} it is clear that $h_{i}^n(t)$ are hypergeometric functions although some care is needed in the hypergeometric representation since some of the coefficients vanish (see Lemma~\ref{fothlemma}). If we set $\nep=n+\epsilon$, then for $i>0$ odd we can rewrite formula \eqref{waveodd} as
\begin{equation}\label{hypgeomodd}
\begin{split}
h_{n-i}^n(t)=&\frac{(-1)^n\sqrt{2n-2i+1}\,(n+1)}{(-n)_i}\prod_{m=0}^{\frac{i-3}{2}}(n-1-2m)^2\\
&\qquad\times\lim_{\epsilon\rightarrow0}{}_{i+2}F_{i+1}\left(\begin{matrix} -n,\; \alpha_1,\; \beta_1 \\
1, \;\alpha_0,\; \beta_0 \end{matrix}\,; t\right),
\end{split}
\end{equation}
where $\alpha_1=\{\nep-i+1,\nep-i+3,\dots,\nep+2\}$, $\alpha_0=\{\nep-i+2,\nep-i+4,\dots,\nep+1\}$, $\beta_1=\{-\nep+2,-\nep+4,\dots,-\nep+i-1\}$, $\beta_0=\{-\nep+1,-\nep+3,\dots,-\nep+i-2\}$. A similar formula can be written for $i$ even using \eqref{waveeven}.
\end{rem}

\begin{rem}\label{GHE}
From the theory of generalized hypergeometric functions we know that the function 
$$F=\hypergeom{i}{j}{a_1,\ \dots\ a_i}{b_1,\ \dots\ b_j}{t}$$
satisfies the differential equation
$$\left[D(D+b_1-1)\cdots(D+b_{j}-1)-t(D+a_1)\cdots(D+a_{i})\right]F=0$$
where $D=t\frac{d}{dt}$, see \cite[\S16.8(ii)]{nist}.
Thus if $i$ is odd, $t\in(0,1)$ we can use \eqref{hypgeomodd} to see that $h_{n-i}^n(t)$ satisfies the differential equation
\begin{equation}\label{diffeq}
\big[t(1-t)\frac{d^2}{dt^2}+(1-(i+2)t)\frac{d}{dt}+(n-i+1)(n+2)\big]\mathcal{L}_i h_{n-i}^n(t)=0,
\end{equation}
where 
\begin{equation}
\mathcal{L}_i=\prod_{m=0}^{\frac{i-3}{2}}(D-n+2m)\prod_{m=0}^{\frac{i-1}{2}}(D+n-i+1+2m).
\end{equation}
The use of equation~\eqref{wavecoeff} shows that
\begin{align*}
 \mathcal{L}_i h_{n-i}^n(t)&=(-1)^{n+\frac{i-1}{2}}(-n-1)_i\sqrt{2n-2i+1}
  \,\hypergeom21{-n+i-1,\ n+2}{1}{t}\\&=c_i \hat p^{0,i}_{n-i+1}(1-2t),
\end{align*}
where $\hat p^{0,i}_{n-i+1}(t)$ is the orthonormal Jacobi polynomial on $[-1,1]$. This implies that for fixed $i$, the functions $\mathcal{L}_i h_{n-i}^n(t)$ are orthogonal on $[0,1)$ for $n\ge i$.

Using \cite{lvw} and the specific form of the hypergeometric representation in \eqref{hypgeomodd}, it is possible to derive even lower-order differential equations for $h_{n-i}^n(t)$, but they also correspond to generalized eigenvalue equations.
\end{rem}

\section{Wavelets for the Alpert multiresolution}\label{se4}

As mentioned in section 2 the functions $\Phi_n$ generate a multiresolution analysis with $V_0={\rm cl}_{\LR}{\rm span}\{\phi_j(\cdot-i),\ i\in\Z,\ j=0\ldots,n\}$ where the function $\phi_j(\frac{t+1}{2})$ when restricted to $[-1,1)$ is the normalized Legendre polynomial $\hat p_j(t)$. In equation~\eqref{nest} the wavelet space $W_0$ is the orthogonal complement of $V_0$ in $V_1$. We look for a set $\Psi_n=[\psi^n_0,\cdots, \psi^n_n]^T$ of functions in $\LR$, each supported on $[0,1)$, whose integer translates provide a basis for $W_0$. Using properties (i)-(iv) of the functions $h^{n}_{i}(t)$ and a standard argument (see for instance \cite{alb}) we see that if we choose 
\begin{equation}\label{waveletprop}
\psi^n_j(t)=h^n_j(2t-1)\chi_{[0,1)}(t),
\end{equation}
then we obtain a basis for $\LR$ by dilations and translations.
\begin{theorem}\label{waveident}
The set of functions $\{ 2^{\frac{k}{2}} \psi^n_j(2^k\cdot-i),i,k\in\Z, j=0,\ldots,n \}$ forms an orthonormal, compactly supported, piecewise polynomial basis for $\LR$.  
\end{theorem}

On the interval $[0,1)$ we have
\begin{theorem}\label{waveidentint}
The set of functions $\{ \sqrt{2}\phi_j,\ 2^{\frac{k}{2}} \psi^n_j(2^k\cdot-i),k\in\Z^+,\ i=0,\ldots, 2^k-1,\ j=0,\ldots,n \}$ forms an orthonormal, piecewise polynomial basis for $\Ltwo$.  
\end{theorem}

From the theory of multiresolution analysis, there are $(n+1)\times(n+1)$ matrices $D^n_{-1}$ and $D^n_1$ 
such that the following equations
\begin{align}\label{wrefine}
\Psi_n\left(\frac{t+1}{2}\right)&=D^n_{-1}\Phi_n(t+1)+D^n_1\Phi_n(t)\nonumber\\&=D^n_{-1}P_n(2t+1)\chi_{[-1,0)}(t)+D^n_1P_n(2t-1)\chi_{[0,1)}(t)
\end{align} 
hold, where
$$
D^n_1=2\int_0^1\Psi_n\left(\frac{t+1}{2}\right) P_n(2t-1) dt,
$$
and
$$
D^n_{-1}=2\int_{-1}^0\Psi_n\left(\frac{t+1}{2}\right)P_n(2t+1) dt.
$$
Thus
$$
(D^n_1)_{i,j}=2\int_0^1\psi^n_i\left(\frac{t+1}{2}\right)\hat p_j(2t-1) dt=2\int_0^1 h^n_i(t)\hat p_j(2t-1) dt,\ 0\le i,j\le n .
$$
These formulas and property (v) show that $D^n_1$ is an upper triangular matrix and as we will see below the diagonal entries are positive. 
From the orthogonality properties of the above functions we find
\begin{equation}\label{waveref}
4I_n=D_{-1}^n(D_{-1}^n)^T+D_{1}^n(D_{1}^n)^T.
\end{equation}
We note that the right-hand side of the above equation differs by a factor of 2 from equation~(5) in \cite{gm}. This is due to the fact that we have normalized the components of $\Psi_n$ to be orthonormal. The symmetry properties of the wavelet functions give
\begin{equation}\label{symmd}
 (D^n_{-1})_{i,j}=(-1)^{i+j+1}(D^n_1)_{i,j},
\end{equation}
which combined with \eqref{waveref} leads to the orthogonality relations
\begin{equation}\label{delementorthogonal}
 0=((-1)^{i+k}+1)\sum_{j=i}^n(D^n_1)_{i,j}(D^n_1)_{k,j}, \ k<i,
\end{equation}
and
\begin{equation}\label{denormalization}
 2=\sum_{j=i}^n(D^n_1)_{i,j}(D^n_1)_{i,j}.
\end{equation}
Using the representation developed above for $\psi^n_j$ we will show,
\begin{theorem}\label{wcoeff} The nonzero entries in $D^n_1$ are given as follows:
\begin{equation}\label{dii}
 (D^n_1)_{n,n}=\sqrt{2},
\end{equation}
for $i$ odd and $j\le i$,
 \begin{align}\label{dd}
(D^n_1)_{n-i,n-j}&=c_{n,i,j}(-1)^{n+j}\sqrt{2(n-i)+1}\sqrt{2(n-j)+1}\nonumber\\
&\times\hypergeom43{-\frac{i-j}{2},\
\frac{j-i+1}{2},\ n-\frac{i+j-1}{2},\ n-\frac{i+j}{2}+1
}{n-\frac{i}{2}+\frac{3}{2},\ -\frac{i}{2},\ n-i+\frac{3}{2}}{1},
\end{align}
where
\begin{align*}
c_{n,i,j}&=2^{i/2}\,\frac{i!!(n-i+\frac{3}{2})_{\frac{i-1}{2}}(-n+j)_{n-i}(n-j+1)_{n-i}}{(2n-i+1)!},
\end{align*}
while for $i$ even and positive and $j\le i$,
\begin{equation}\label{oded}
(D^n_1)_{n-i,n-j}=(D^{n-1}_1)_{n-i,n-j}.
\end{equation}
\end{theorem}
\begin{rem}
Note that the ${}_4 F_3$ hypergeometric functions are balanced and satisfy the orthogonality equations \eqref{delementorthogonal}-\eqref{denormalization}. From the explicit formulas above we see that the diagonal entries are positive. As shown in \cite[Theorem~4]{gm} the orthogonality, upper triangularity, and positivity of the diagonal entries uniquely specify the matrix $D^n_1$.
\end{rem}

\begin{proof}
We prove the result for $i=0$ and $i$ odd since for $i>0$ even the formula follows from the properties of the wavelets. For $i=0$ or $i$ odd equation~\eqref{hyperone} in Theorem~\ref{eqwave} and condition (v) give,
\begin{equation}
(D^n_1)_{n-i,n-j}=(-1)^{n-j}\sqrt{2}\sqrt{2(n-j)+1}\sum_{s=n-i}^{n-j}\sum_{k=0}^n\frac{(-n+j)_s(n-j+1)_s}{(1)_s(1)_s}\frac{d^i_{n,k}}{k+s+1}.
\end{equation}
Lemma~\ref{lang} yields
$$
\sum_{k=0}^n\frac{(-n)_k(n-i+1)_k}{(1)_k(1)_k}\frac{1}{k+s+1}=\frac{n!(-s)_{n-i}}{(n-i)!(s+1)_{n+1}}=(-1)^{i}\frac{(-n)_{i}(-s)_{n-i}}{(s+1)_{n+1}}
$$
which coupled with Lemma~\ref{sothlemma} shows that
$$
\sum_{k=0}^n\frac{d^i_{n,k}}{k+s+1}=\frac{(-1)^{n+i}\sqrt{2n-2i+1}\,(-s)_{n-i}}{(s+1)_{n+1}}\prod_{m=0}^{\frac{i-1}{2}}(n-s-2m)\prod_{m=0}^{\frac{i-3}{2}}(n+s-2m).
$$
Thus
\begin{align*}
(D^n_1)_{n-i,n-j}&=(-1)^{n+j}\sqrt{2}\sqrt{2(n-i)+1}\sqrt{2(n-j)+1}\sum_{s=n-i}^{n-j}\frac{(-n+j)_s(n-j+1)_s}{(s+n+1)!(s-n+i)!}\\
&\times\prod_{m=0}^{\frac{i-1}{2}}(n-s-2m)\prod_{m=0}^{\frac{i-3}{2}}(n+s-2m),
\end{align*}
where the fact that
$\frac{(-s)_{n-i}}{(1)_s(1)_s(s+1)_{n+1}}=\frac{(-1)^{n-i}}{(s+n+1)!(s-n+i)!}$ has been used. Making the change of variable $s=n-i+l$ yields,
\begin{align*}
(D^n_1)_{n-i,n-j}&=(-1)^{n+j}\sqrt{2}\sqrt{2(n-i)+1}\sqrt{2(n-j)+1}\frac{(-n+j)_{n-i}(n-j+1)_{n-i}}{(2n-i+1)!}\\
&\times\sum_{l=0}^{i-j}\frac{(j-i)_l(2n-i-j+1)_l}{(2n-i+2)_l l!}\prod_{m=0}^{\frac{i-1}{2}}(i-l-2m)\prod_{m=0}^{\frac{i-3}{2}}(2n-i+l-2m).
\end{align*}
Note that the terms in the above sum are zero if $l$ is odd (here $i$ is odd) since the
first product starts from $i-l$ which is positive and even and ends up with 
$-l+1$ which is negative or zero. Using the identity
$$
(a)_{2l}=2^{2l}\left(\frac{a}{2}\right)_l\left(\frac{a+1}{2}\right)_l
$$
for $(j-i)_{2l}$, $(2l)!$, $(2n-i+2)_{2l}$ and $(2n-i-j+1)_{2l}$
together with the identities,
\begin{align*}
&\frac{\prod_{m=0}^{\frac{i-1}{2}}(i-2l-2m)}{(\frac{1}{2})_l}
=\frac{i!!}{(\frac{-i}{2})_l}
\end{align*}
and
\begin{align*}
&\frac{\prod_{m=0}^{\frac{i-3}{2}}(2n-i+2l-2m)}{(n-\frac{i}{2}+1)_l}=
2^{\frac{i-1}{2}}\frac{(n-i+\frac{3}{2})_{\frac{i-1}{2}}}{(n-i+\frac{3}{2})_l}
\end{align*}
give the result.

\end{proof}

Because of the way that the indices enter in the above formulas it is not so simple to obtain recurrence relations. For this reason we obtain another representation.
\begin{lemma}\label{altd} Suppose $j<i$. For $i$ odd and $j$ even,
 \begin{align}\label{dwioje}
(D^n_1)_{n-i,n-j}&=(-1)^{\frac{i-j+1}{2}}\frac{(j+1)!!(i-j-2)!!}{2^{i/2}(n-\frac{i+j-1}{2})_{\frac{i+1}{2}}}\sqrt{2(n-i)+1}\sqrt{2(n-j)+1}\nonumber\\
&\times\hypergeom43{-\frac{j}{2},\
\frac{j-i+1}{2},\ n+\frac{3-j}{2},\ -n+\frac{i+j}{2}
}{1,\ \frac{1}{2},\ \frac{3}{2}}{1},
\end{align}
for $i,\ j$ odd,
 \begin{align}\label{dwiojo}
&(D^n_1)_{n-i,n-j}=(-1)^{\frac{i-j}{2}+1}\frac{(j)!!(i-j-1)!!}{2^{\frac{i+2}{2}}(n-\frac{i+j-2}{2})_{\frac{i+1}{2}}}\sqrt{2(n-i)+1}\sqrt{2(n-j)+1}\nonumber\\
&\quad\times(2n-j+2)(j+1)\,\hypergeom43{-\frac{j-1}{2},\
\frac{j-i}{2}+1,\ n-\frac{j}{2}+2,\ -n+\frac{i+j+1}{2}
}{2,\ \frac{3}{2},\ \frac{3}{2}}{1},
\end{align}
while for the remaining coefficients use equation~\eqref{oded}.
\end{lemma}
\begin{proof}
We only consider the case when $i$ is odd and $j$ is even. 
The Whipple transformation of a balanced ${}_4F_3$ hypergeometric function is the following
\begin{align*}
&\hypergeom43{-n,\ x,\ y,\ z}{u,\ v,\ w}{1}\\&=\frac{(1-v+z-n)_n(1-w+z-n)_n}{(v)_n(w)_n}\hypergeom43{-n,\
  u-x,\ u-y,\ z}{u,\ 1-v+z-n,\ 1-w+z-n}{1}.
\end{align*}
Thus with $i$ odd, $j$ even, $n=\frac{i-j-1}{2}$, $u=-i/2$, and
$z= n-\frac{i+j}{2}+1$, 
we find
\begin{align*}
&\hypergeom43{-\frac{i-j}{2},\
\frac{j-i+1}{2},\ n-\frac{i+j-1}{2},\ n-\frac{i+j}{2}+1
}{n-\frac{i}{2}+\frac{3}{2},\ -\frac{i}{2},\ n-i+\frac{3}{2}}{1}\\&=\frac{(1)_{\frac{i-j-1}{2}}(-\frac{i}{2}+1)_{\frac{i-j-1}{2}}}{(n-\frac{i-3}{2})_{\frac{i-j-1}{2}}(n-i+\frac{3}{2})_{\frac{i-j-1}{2}}}\hypergeom43{-\frac{j}{2},\
\frac{j-i+1}{2},\ -n+\frac{j-1}{2},\ n-\frac{i+j}{2}+1
}{1,\ -\frac{i}{2},\ -\frac{i}{2}+1}{1}.
\end{align*}
An application of this formula to equation~\eqref{dd} and the identity 
$$
\left(n-\frac{i+j-1}{2}\right)_{\frac{j+2}{2}}\left(n-\frac{i-3}{2}\right)_{\frac{i-j-1}{2}}=\left(n-\frac{i+j-1}{2}\right)_{\frac{i+1}{2}},
$$ 
yields
\begin{align}\label{dwipone}
(D^n_1)_{n-i,n-j}&=(-1)^{j-i}\sqrt{2}\frac{2^{\frac{i-3-2j}{2}}i!!(1)_{\frac{i-j-1}{2}}(-\frac{i}{2}+1)_{\frac{i-j-1}{2}}}{(n-\frac{i+j-1}{2})_{\frac{i+1}{2}}(i-j)!}\sqrt{2(n-i)+1}\sqrt{2(n-j)+1}\nonumber\\
&\times\hypergeom43{-\frac{j}{2},\
\frac{j-i+1}{2},\ -n+\frac{j-1}{2},\ n-\frac{i+j}{2}+1
}{1,\ -\frac{i}{2},\ -\frac{i}{2}+1}{1}.
\end{align}
Another use of Whipple's transformation with $n= \frac{j}{2}$, $z=\frac{j-i+1}{2}$, and $u=1$ gives the result.
\end{proof}

\begin{rem}
Besides the orthogonality relations \eqref{delementorthogonal}-\eqref{denormalization}, equations~\eqref{nest} and \eqref{refsym} imply the relations,
\begin{equation}\label{mc1ortho}
2I_{n+1}=C^n_{-1}{C^n_{-1}}^T+C^n_1{C^n_1}^T,
\end{equation}
and
\begin{equation}\label{cdortho}
 C^n_{-1} {D^n_{-1}}^T + C^n_1 {D^n_{1}}^T=0.
\end{equation}
Using the symmetry property \eqref{symmd} and equation~(15) of \cite{gm} we see that that the matrix $A$ composed of the even rows of $C^n_1$ and the odd rows of $\frac{D^n_1}{\sqrt{2}}$ or vice versa is unitary which yields orthogonality relations among the entries of $C^n_1$, and $\frac{D^n_1}{\sqrt{2}}$. 
It is interesting to compare these orthogonality relations with other known orthogonality relations for ${}_4F_3$ series, and in particular with the orthogonality of the Wigner $6j$-Symbols and Racah polynomials, see \cite{wil} or the book \cite{nsu}. However, we could not relate the orthogonality relations above to this theory. Providing a Lie-theoretic interpretation of these new orthogonality equations is a very interesting problem. Another challenging problem is to connect the orthogonality relations here to an appropriate extension of the Fields and Wimp expansion formula \cite{fw}.
\end{rem}

\section{Fourier Transform}\label{se5}
An important tool in wavelet theory is the Fourier transform given by,
\begin{equation}\label{wft}
\hat\psi^n_k(\theta)=\int_{0}^1\psi^n_k(t)e^{-i\theta t} dt.
\end{equation}
In order to compute this Fourier transform we will focus on 
\begin{equation}\label{wfh}
\hat h^n_k(\theta)=\int_{-1}^1 h^n_k(t)e^{-i\theta t} dt,
\end{equation}
since equation \eqref{waveletprop} implies $\hat{\psi}^{n}_{k}(\theta)=\frac{e^{-i\theta/2}}{2}\hat{h}^n_k(\theta/2)$.
\begin{theorem}\label{wavetrans}
For $j$ odd $\hat h^n_k(\theta)$ is given by
\begin{align}\label{transeqodd}
 \hat h^n_{n-j}(\theta)=&\frac{(-1)^{\frac{j+1}{2}}2^{j+1}(\frac{j+1}{2})!(n+1)!(n+2)!}{(j+2)!(2n+3)!(n+\frac{3}{2}-\frac{j}{2})!}\sqrt{2n-2j+1}(-i\theta)^{n+2}\nonumber\\&
\qquad\times\hypergeom23{\frac{n+3}{2},\ \frac{n+4}{2}}{\frac{j+4}{2},\ n+\frac{5}{2},\ n+\frac{5-j}{2}}{-\frac{\theta^2}{4}},
\end{align}
while for $j$ even, $\hat h^n_{n-j}(\theta)=\hat h^{n-1}_{n-j}(\theta)$.
\end{theorem}
\begin{rem}
It is remarkable that while the wavelet functions are limits of higher-order hypergeometric functions, their Fourier transforms have a simple closed formula in terms of ${}_2F_3$ series. Since $h^n_k(t)$ is a polynomial of degree $n$ on $[0,1)$, Euler's formula and integration by parts show that $\hat h^n_{n-j}(\theta)$ can be written in terms of $\sin{\theta}$ and $\cos{\theta}$ multiplied by polynomials in $1/\theta$ of degree $n+1$. This is also true of the wavelets constructed by Alpert et al.
\end{rem}
\begin{proof}
We prove the formula for $j=0$ or $j$ odd since for $j$ even and positive the result follows from equation~\eqref{waveeven}. From the symmetry properties of $h^n_{n-j}$ we find
$$
\int_{-1}^1 h^n_{n-j}(t)f(t)dt=\int_{0}^1[(-1)^{n-j+1}f(-t)+f(t)] h^n_{n-j}(t)dt,
$$
so that
\begin{equation}\label{wftsy}
\hat h^n_{n-j}(\theta)=\int_{0}^1[(-1)^{n-j+1}e^{i\theta t}+e^{-i\theta t}] h^n_{n-j}(t)dt.
\end{equation}
From  property (ii),
$
\int_{-1}^1 h^n_k(t) t^s dt=0,
$ for $0\le s\le n$. Thus only the moments for $s>n$ and $s\equiv n-j+1\mod2$ are nonzero. For such $s$ we find
\begin{align*}
&\int_0^1 h^n_{n-j}(t) t^s dt=\sum_{k=0}^n d^j_{n,k}\frac{1}{k+s+1}\\&=(-1)^n\frac{\sqrt{2n-2j+1}}{(-n)_j}\sum_{k=0}^n\frac{(-n)_k(n-j+1)_k}{(k!)^2}
\frac{\prod_{m=0}^{\frac{j-1}{2}}(n+k+1-2m)\prod_{m=0}^{\frac{i-3}{2}}(n-k-1-2m)}{k+s+1}.
\end{align*}
By Lemma~\ref{sothlemma},  $k$ can be replaced by $-(s+1)$ in the above products
which gives
\begin{align*}
\int_0^1 h^n_{n-j}(t) t^s dt&=(-1)^n\frac{\sqrt{2n-2j+1}}{(-n)_j}\prod_{m=0}^{\frac{j-1}{2}}(n-s-2m)\prod_{m=0}^{\frac{j-3}{2}}(n+s-2m)\\&\times\sum_{k=0}^n\frac{(-n)_k(n-j+1)_k}{(k!)^2}\frac{1}{k+s+1}.
\end{align*}
The use of Lemma~\ref{lang} with $x=s+1$ shows that the above sum is equal to $\frac{(-n)_j(-1)^j(-s)_{n-j}}{(s+1)_{n+1}}$ so that
\begin{equation*}
\int_0^1 h^n_{n-j}(t) t^s dt=(-1)^{n+j}\frac{\sqrt{2n-2j+1}(-s)_{n-j}}{(s+1)_{n+1}}\prod_{m=0}^{\frac{j-1}{2}}(n-s-2m)\prod_{m=0}^{\frac{j-3}{2}}(n+s-2m).
\end{equation*}
The identity $\frac{(-s)_{n-j}(-1)^{n+j}}{s!}=\frac{1}{(s-n+j)!}$ allows the above integral to be rewritten as
\begin{align}\label{firstint}
\int_0^1 h^n_{n-j}(t)\frac{t^s}{s!} dt=\frac{\sqrt{2n-2j+1}}{(s-n+j)!(s+1)_{n+1}}\prod_{m=0}^{\frac{j-1}{2}}(n-s-2m)\prod_{m=0}^{\frac{j-3}{2}}(n+s-2m).
\end{align} 
For $j>0$  the change of variables $s=n+2+2r,\ r=0,1,\ldots$ yields
\begin{align*}
&\int_{-1}^1 h^n_{n-j}(t)\frac{t^{n+2+2r}}{(n+2+2r)!} dt\\
&=\frac{2^{j+1}(-1)^{\frac{j+1}{2}}\sqrt{2n-2j+1}}{(2+2r+j)!(n+3+2r)_{n+1}}\prod_{m=0}^{\frac{j-1}{2}}(m+r+1)\prod_{m=0}^{\frac{j-3}{2}}(n+1+r-m).
\end{align*}
This can be simplified using $\frac{1}{(2+j+2r
)!}=\frac{1}{(2+j)!(3+j)_{2r}}=\frac{1}{(2+j)!(\frac{3+j}{2})_r(\frac{4+j}{2})_r 2^{2r}}$,
\begin{align*}
\frac{1}{(n+3+2r)_{n+1}}
=\frac{(n+2)!}{(2n+3)!}\frac{(\frac{n+3}{2})_r(\frac{n+4}{2})_{r}}{(n+2)_r(n+\frac{5}{2})_r},
\end{align*}
$$
\prod_{m=0}^{\frac{j-1}{2}}(m+r+1)=\frac{(\frac{j+1}{2}+r)!}{r!}=\frac{(\frac{j+1}{2})!(\frac{j+3}{2})_r}{r!},
$$
and
$$
\prod_{m=0}^{\frac{j-3}{2}}(n+r+1-m)=\frac{(n+1+r)!}{(n+\frac{3}{2}-\frac{j}{2}+r)!}=\frac{(n+1)!(n+2)_r}{(n+\frac{3}{2}-\frac{j}{2})!(n+\frac{5}{2}-\frac{j}{2})_r},
$$
to obtain
\begin{align}\label{eqtrans}
&\int_{-1}^1 h^n_{n-j}(t)\frac{t^{n+2+2r}}{(n+2+2r)!} dt\\\nonumber
\qquad&=\frac{(-1)^{\frac{j+1}{2}}2^{j+1}(\frac{j+1}{2})!(n+1)!(n+2)!}{(j+2)!(2n+3)!(n+\frac{3}{2}-\frac{j}{2})!}\sqrt{2n-2j+1}\frac{(\frac{n+3}{2})_r(\frac{n+4}{2})_{r}}{r!(\frac{4+j}{2})_r(n+\frac{5}{2})_r(n+\frac{5}{2}-\frac{j}{2})_r}\frac{1}{4^r}.
\end{align}
Thus
\begin{align*}
\hat h^n_{n-j}(\theta)&=\sum_{s=0}^{\infty}\int_{-1}^1 h^n_{n-j}(t)(-i\theta)^s \frac{t^s}{s!} dt\\&=(-i\theta)^{n+2}\sum_{r=0}^{\infty}(-i\theta)^{2r}\int_{-1}^1 h^n_{n-j}(t)\frac{t^{n+2+2r}}{(n+2+2r)!} dt
\end{align*}
and the result for $j$ odd is obtained by the substitution of \eqref{eqtrans} in the above formula. For $j=0$ the substitution $s=n+1+2r$ in equation~\eqref{firstint} shows that $\hat h^n_n=\hat h^{n-1}_n$ which completes the proof.
\end{proof}

\begin{rem}
From the differential equation given in Remark~\ref{GHE} and 
\eqref{transeqodd} we obtain differential equations for
$\hat{h}^n_{n-j}$. With
$D_{\theta}:=\frac{\theta}{2}\frac{d}{d\theta}=\theta^2
\frac{d}{d\theta^2}$ and since $D_{\theta}\theta^n f(\theta)=\frac{n}{2} \theta^n f(\theta)+\theta^n D_{\theta}f(\theta)$ we find
for $j$ odd
\begin{align}\label{psihatdfeqod}
&\bigg[\left(D_{\theta}-\frac{n+2}{2}\right)\left(D_{\theta}+\frac{n+1}{2}\right)\left(D_{\theta}+\frac{j-n}{2}\right)\left(D_{\theta}+\frac{n+1-j}{2}\right)\\\nonumber&+\frac{\theta^2}{4}\left(D_{\theta}+\frac{1}{2}\right)\left(D_{\theta}+1\right)\bigg]\hat{h}^n_{n-j}(\theta)=0,
\end{align}
while for $j$ even replace $n$ by $n-1$ and $j$ by $j-1$ in the above formula.
\end{rem}

\begin{rem}
The formula for $\hat h^n_{n-j}$ allows the development of asymptotic formulas
for large $n$ and $j$. We will not systematically explore this here but be
content to give a simple example in the case when $j=tn$, $0<t\le 1$. In this
case for $j$ odd
\begin{align}\label{transeqoddscaled}
 \hat h^n_{(1-t)n}(\theta)=&\frac{(-1)^{\frac{tn+1}{2}}2^{tn+1}(\frac{tn+1}{2})!(n+1)!(n+2)!}{(tn+2)!(2n+3)!(n+\frac{3}{2}-\frac{tn}{2})!}\sqrt{2n-2tn+1}(-i\theta)^{n+2}\nonumber\\&
\qquad\times\hypergeom23{\frac{n+3}{2},\ \frac{n+4}{2}}{\frac{tn+4}{2},\ n+\frac{5}{2},\ n+\frac{5-tn}{2}}{-\frac{\theta^2}{4}}.
\end{align}
The use of Stirling's formula (with the help of Mathematica) shows that 
\begin{align*}
&\frac{(\frac{tn+1}{2})!(n+1)!(n+2)!}{(tn+2)!(2n+3)!(n+\frac{3}{2}-\frac{tn}{2})!}\nonumber\\&=e^n(4n)^{-n}\left(1-\frac{t}{2}\right)^{-n(1-\frac{t}{2})-2}\frac{(2t)^{-\frac{nt}{2}}}{16\sqrt{2}t^{3/2}n^3}\left(1-\frac{104-12t+27t^2}{24(2-t)nt}+O\left(\frac{1}{n^2}\right)\right).
\end{align*}
Since the hypergeometric function can be expanded as
$$
\hypergeom23{\frac{n+3}{2},\ \frac{n+4}{2}}{\frac{tn+4}{2},\ n+\frac{5}{2},\ n+\frac{5-tn}{2}}{-\frac{\theta^2}{4}}=1-\frac{\theta^2}{4nt(2-t)}+O\left(\frac{1}{n^2}\right),
$$
we find
\begin{align*}
\hat h^n_{(1-t)n}(\theta)=\frac{(-1)^{\frac{tn+1}{2}}(-i\theta)^{n+2}e^n\sqrt{2n(1-t)+1}}
{2^{n+3/2} n^{n+3}t^{\frac{nt+3}{2}} (2-t)^{n(1-\frac{t}{2})+2}}
\left(1-\frac{104-12t+27t^2+6\theta^2}{24(2-t)tn}+O\left(\frac{1}{n^2}\right)\right).
\end{align*}
\end{rem}

\section{Recurrence Formulas}\label{se6}
The formulas for the entries in $D^n_1$ given by Lemma~\ref{altd} allow simple recurrences in $n$ and $i$. To this end we use formula (3.7.8) in \cite{aar} for balanced ${}_4 F_3$ series:
\begin{align}\label{4f3recc}
 &\frac{b(e-a)(f-a)(g-a)}{a-b-1}(F(a-,b+)-F)\\\nonumber&-\frac{a(e-b)(f-b)(g-b)}{b-a-1}(F(a+,b-)-F)+cd(a-b)F=0,
\end{align}
where $F=\hypergeom43{a,\ b,\ c,\ d}{e,\ f,\ g}{1}$, $F(a+,b-)=\hypergeom43{a+1,\ b-1,\ c,\ d}{e,\ f,\ g}{1}$ and
$F(a-,b+)=\hypergeom43{a-1,\ b+1,\ c,\ d}{e,\ f,\ g}{1}$. From this equation we find:
\begin{theorem}\label{recini} The entries of the matrix $D^{n}_1$ satisfy the following recurrence relations
\begin{equation}\label{drec}
(D^n_1)_{n-i+2,n-j}+k^1_{n,i,j}(D^n_1)_{n-i,n-j}+k^2_{n,i,j}(D^n_1)_{n-i-2,n-j}=0,
\end{equation}
where $i$ is odd, $j<i-2$,
\begin{align}
k^1_{n,i,j}&= h\left(\left(cd(a-b)-\frac{b(e - a)(f - a)(g - a)}{(a - b - 1)}\right)\frac{(1+a-b)}{a(e-b)(f-b)(g-b)}-1\right),\label{k1ioje}\\
k^2_{n,i,j}&= l\frac{b(e - a)(f - a)(g - a)(1+a-b)}{(a - b - 1)a(e-b)(f-b)(g-b)},\label{k2ioje}
\end{align}
and
\begin{align*} 
&a=\frac{j-i+1}{2},\quad b= -n+\frac{i+j}{2},\quad c=-\frac{j}{2},\quad d=n+\frac{3-j}{2}, \quad e=1,\quad f=\frac{1}{2},\\ 
&g=\frac{3}{2},\quad  h=\frac{\sqrt{2(n-i)+5}(-2n+i+j-1)}{\sqrt{2(n-i)+1}(i-j-2)}, \\
&l=\frac{\sqrt{2(n-i)+5}(2n-i-j-1)(2n-i-j+1)}{\sqrt{2(n-i)-3}(i-j)(i-j-2)}.
\end{align*}
\end{theorem}

\begin{proof}
If $j$ is even, the recurrence relation \eqref{drec} follows from equations \eqref{dwioje} and \eqref{4f3recc}. If $j$ is odd, we use equations \eqref{dwiojo} and \eqref{4f3recc} to deduce that \eqref{drec} holds with $k^1$ and $k^2$ defined by equations \eqref{k1ioje}-\eqref{k2ioje} where $a=\frac{j-i+2}{2}$, $b= -n+\frac{i+j+1}{2}$, $c=-\frac{j-1}{2}$, $d=n+\frac{4-j}{2}$, $e=2$, $f=\frac{3}{2}$, $g=\frac{3}{2}$, $h=\frac{\sqrt{2(n-i)+5}(-2n+i+j-2)}{\sqrt{2(n-i)+1}(i-j-1)}$, and $l=\frac{\sqrt{2(n-i)+5}(2n-i-j)(2n-i-j+2)}{\sqrt{2(n-i)-3}(i-j+1)(i-j-1)}$. Although $a$, $b$, $c$, $d$, $e$, $f$, $g$, $h$ and $l$ are different in this case, they lead to the same formulas for $k^1$ and $k^2$.
\end{proof}

Likewise we can obtain a recurrence relation in $n$.

\begin{theorem}\label{recinn} The entries of the matrix $D^{n}_1$ satisfy the following recurrence relations
\begin{equation}\label{drecn}
(D^{n+1}_1)_{n-i+1,n-j+1}+k^1_{n,i,j}(D^n_1)_{n-i,n-j}+k^2_{n,i,j}(D^{n-1}_1)_{n-i-1,n-j-1}=0,
\end{equation}
where $i$ is odd, $j<i<n$, $k^1$ and $k^2$ are given in equations \eqref{k1ioje}-\eqref{k2ioje}, and 
\begin{align*} 
&a=n+\frac{3-j}{2},\quad b= -n+\frac{i+j}{2},\quad c=-\frac{j}{2},\quad d=\frac{j-i+1}{2},\quad e=1,\quad f=\frac{1}{2},\\ 
&g=\frac{3}{2}, \quad h=\frac{\sqrt{2(n-i)+3}\sqrt{2(n-j)+3}(2n-i-j+1)}{\sqrt{2(n-i)+1}\sqrt{2(n-j)+1}(2n-j+2)},\\
&l=\frac{\sqrt{2(n-i)+3}\sqrt{2(n-j)+3}(2n-i-j-1)(2n-i-j+1)}{\sqrt{2(n-i)-1}\sqrt{2(n-j)-1}(2n-j)(2n-j+2)}.
\end{align*}
\end{theorem}

\section*{Acknowledgements}
We would like to thank a referee for carefully reading the manuscript and the
observation that the recurrence coefficients in equations~\eqref{drec} and
\eqref{drecn} are the same for $j$ even and odd.

\end{document}